\documentclass[10pt]{article}
\font\smallit=cmti10

\usepackage{amssymb,latexsym,amsmath,epsfig,amsthm} %

\usepackage{setspace}
\usepackage{enumitem}
\usepackage[table]{xcolor}
\usepackage{nicefrac} 
\usepackage{graphicx,bm,mathtools}
\usepackage{cite,url}
\usepackage{subcaption}
\usepackage{caption}
\usepackage[ total={6.5in, 8.5in}]{geometry}
\usepackage[normalem]{ulem}
\usepackage{multirow}
\usepackage{lmodern}
\usepackage{algorithm}
\usepackage{algpseudocode}

\makeatletter

\renewcommand\section{\@startsection {section}{1}{\z@}
{-30pt \@plus -1ex \@minus -.2ex}
{2.3ex \@plus.2ex}
{\normalfont\normalsize\bfseries\boldmath}}

\renewcommand\subsection{\@startsection{subsection}{2}{\z@}
{-3.25ex\@plus -1ex \@minus -.2ex}
{1.5ex \@plus .2ex}
{\normalfont\normalsize\bfseries\boldmath}}

\renewcommand{\@seccntformat}[1]{\csname the#1\endcsname. }

\makeatother

\newtheorem{theorem}{Theorem}
\newtheorem{lemma}{Lemma}
\newtheorem{proposition}[theorem]{Proposition}
\newtheorem{corollary}{Corollary}

\theoremstyle{definition}
\newtheorem{definition}{Definition}
\newtheorem{conjecture}{Conjecture}
\newtheorem{open}{Open Problem}

\newtheorem{example}{Example}

\newcommand{\ruleset}[1]{\textsc{#1}}

\newcommand{\SN}[2]{\ruleset{SN\ensuremath{\left(#1,#2\right)}}}
\newcommand{\SNr}[2]{\ruleset{SN\ensuremath{_{\tilde{\mathcal{G}}}\left(#1,#2\right)}}}
\def\sp{\ensuremath{\Sigma(\p)}}
\def\spp{\ensuremath{\Sigma(\p')}}

\def\N{\ensuremath{\mathcal{N}}}
\def\P{\ensuremath{\mathcal{P}}}
\def\T{\ensuremath{\mathcal{T}}}
\def\p{\ensuremath{\bm{p}}}

\def\q{\ensuremath{\bm{q}}}

\def\Nim{\ruleset{Nim}}
\def\SAN{\ruleset{Slow$A$-Nim}}

\begin{document}
							
\begin{center}
\uppercase{\bf \boldmath On the \P-positions of some infinite families of \ruleset{Slow$A$-Nim}}
\vskip 20pt
{\bf Matthieu Dufour}\\
{\smallit Department of Mathematics, University of Quebec, Montreal, Quebec, Canada}\\
{\tt dufour.matthieu@uqam.ca}\\ 
\vskip 10pt
{\bf Silvia Heubach}\\
{\smallit Department of Mathematics, Cal State LA, Los Angeles, USA}\\
{\tt sheubac@mcalstatela.edu}\\ 
\end{center}
\vskip 20pt

\centerline{\bf Abstract}

\noindent We introduce the game \ruleset{Slow$A$-Nim} which generalizes a number of recently studied games. \ruleset{Slow$A$-Nim} is played on $n$ stacks of tokens, and the set $A$ indicates the number of stacks a player can play on.  Once a player has decided on the number $a$ of stacks, s/he will select any $a$ stacks and then remove one token from each stack. The last player to move wins. We give results on the \P-positions of \ruleset{Slow$A$-Nim} for several infinite families. The results for $A = \{n-1\}$, which is the game \ruleset{Slow Exact $k$-Nim} for $k=n-1$ extend recent results for small values of $n$. The other two families,  $A=\{n-1,n\}$ and  $A=\{1,n\}$ have not been previously studied. The \P-positions for  $A = \{n-1\}$ and $A = \{n-1,n\}$ are closely related and have a very elegant description in terms of reduced positions, that is, positions for which unplayable tokens are disregarded. We also provide some general results that will be useful in the study of other sets $A$.\\


\section{Introduction}\label{intro}

The game of \Nim, introduced and solved in 1901 by  Charles L. Bouton~\cite{Bouton:1901} is a two-player game on  $n$ stacks of tokens. The players alternate taking one or more tokens from a single stack, and the player who cannot make a move loses. \Nim~is an impartial combinatorial game, where all possible moves and positions of the game are known (there is no randomness) and  both players have the same moves available from a given position (unlike in Chess). In impartial combinatorial games, positions fall into one of two outcome classes, the $\N$- and the $\P$-positions. Assuming optimal play by both players, a player can win if starting from an $\N$-position (by always moving to a \P-position), while a player starting in a \P-position is bound to lose against a wise adversary.  Therefore, knowing the set of $\P$-positions suffices to determine who will win the game and what the winning moves are. The $\P$-positions are characterized by the well-known property of finite impartial games, namely that there is no move from a \P-position to another \P-position, and that there is at least one move from every \N-position to a \P-position. Alternatively, \P-positions are those positions that have a Sprague-Grundy value of zero. For these and other basic properties of impartial combinatorial games see, for example,\cite{AlbNowWol2019} or \cite{WinningWays}. 

The game of \Nim~started the field of combinatorial games, and due to its central role in impartial games, there are a large number of variations of \Nim. We will mention only those that are directly relevant to our results, namely \ruleset{Moore's} $k$-\Nim \cite{Moo910} and \ruleset{Exact} $k$-\Nim ~\cite{BGHMM15}. In the first game,  a player is allowed to play on up to $k$ stacks, while in the second game the player has to play on exactly $k$ stacks. In each case, play means removing at least one token from among the selected stack(s), and as many as all tokens. 
Recently, slow versions of these two games were investigated. In the slow version of a game, exactly one token is removed from the selected stack(s). We will study a generalized version of these slow \Nim~games, namely the game \ruleset{Slow$A$-Nim}, and extend and complement prior results.\\

\begin{definition}
    The game \ruleset{Slow$A$-Nim} is played on $n$ stacks of tokens, where the set $A \subset \{1,2, \ldots,n\}$ determines on how many stacks a player is allowed to play. In each move, a player selects  $a\in A$  stacks and then removes exactly one token from each of the selected stacks. The number of stacks being played on can change from move to move and from player to player. Note that even when a stack is reduced to zero tokens, thus reducing the number of playable stacks,  we will still regard it as a stack, so the value of $n$ remains fixed for a particular game.  We assume the normal play convention, where the last player able to make a move wins.
\end{definition} 

Clearly, \SAN~with $A=\{1,\ldots,k\}$ describes the slow version of \ruleset{Moore's} $k$-\Nim, while $A=\{k\}$ describes the slow version of \ruleset{Exact} $k$-\Nim. Of these two games, \ruleset{Slow Exact} $k$-\Nim~has been studied more extensively, both in the normal and mis\`ere  (last player to move loses) versions. Here is a short summary of what is known about these two games.
  
\begin{itemize}
    \item  \ruleset{Slow Exact $k$-Nim}: The simplest families of \ruleset{Slow Exact $k$-Nim} games are those where $k=1$ or $k=n$, which are trivial as they have a fixed number of moves, and the Sprague-Grundy values toggle between $0$ and $1$ in each move.  For $k=2$, explicit formulas for the Sprague-Grundy values for normal and mis\`ere play were found for $n=2,3$ in \cite{GHHC20}, and an explicit formula for the \P-positions was proved for $n=5,6$ in \cite{ChiGurKno2021}. Additional results have been obtained by Gurvich and differing sets of co-authors, who have  provided an optimal strategy for \ruleset{Slow Exact $(n-1)$-Nim}  for both normal play \cite{GMMV2023} and mis\`ere play \cite{GMMN2023}. They show that their $M$-rule, namely omitting a maximal stack if all stack heights are odd and omitting a smallest even stack otherwise, is optimal in both cases.  They also extend prior results by giving an explicit description of the $\P$-positions for $k=3$ and $n=4$ under the normal play convention in \cite{GMMV2023}. Table~\ref{tab:knownSEkNim} summarizes what is known for this game, where results highlighted in the same color can be found in the same article.

 \begin{table}[!hbt]
$$
\begin{array} {c|cccccc}
n \backslash k&2 & 3& 4&5 &6 &\cdots \\ \hline
3 & \cellcolor{green!25}  G/ G^- & & & & & \\
4 &  \cellcolor{green!25} G/ G^- & \cellcolor{orange!75} \P\text{-pos} & & & & \\
5 &  \cellcolor{blue!25} \P\text{-pos} &?&\cellcolor{orange!75}M\text{-rule}   & & & \\
6 &  \cellcolor{blue!25} \P\text{-pos} & ? & ? &\cellcolor{orange!75}M\text{-rule} && \\
\vdots & ? &\vdots &\vdots & \ddots& \ddots& \\
\end{array}
$$
\caption{Known results for non-trivial ($1<k<n$) \ruleset{Slow Exact} $k$-{\Nim} games for the respective combinations of $n$ and $k$ values: $G/G^-$ indicates  Sprague-Grundy values for normal and mis\`ere play,  $M$-rule indicates winning strategy in normal (and mis\`ere) play, and \P-pos indicates explicit formulas for the \P-positions under normal play. }
\label{tab:knownSEkNim}
\end{table}

\item \ruleset{Slow Moore's} $k$-\Nim~(play on at most $k$ stacks):  Besides the Sprague-Grundy values for the trivial families when $k=1$ or $k=n$, Gurvich et al.~\cite{GHHC20} gave Sprague-Grundy values for $k=2$ and $n = 2,3$ and characterized the \P-positions for the infinite families when $k \in \{ n-2, n-1\}$ and for the specific games when $k=n-3$ and $n=5,6$. 
\end{itemize} 
 
What makes the \SAN~games different from the usual variations of \Nim~is that the game may end while there are still tokens available, namely when the number of non-zero stacks is smaller than $a=\min(A)$.  This makes finding the \P-positions more difficult, which is evident in Table~\ref{tab:knownSEkNim}:  the \P-positions  of \ruleset{Slow Exact $k$-Nim} are known only for small values of $n$ and $k$. We overcome this difficulty by establishing equivalence of \SAN~with a game played on reduced positions with adjusted moves. Our main results are as follows:
\begin{itemize}
    \item We provide a partial structure result on the {\bf \P-positions} of  \SAN~for {\bf general} $\bm{A}$ that highlights the relevance of the number of stacks with an odd stack height {\bf (Proposition~\ref{prop:Game_n_A_Strong})}, then  derive the {\bf \P-positions} of  \SAN~with $\bm{A=\{1,k\}}$ for all odd values of $k$ and some even values of $k$ {\bf(Theorems~\ref{thm:1k}} and~{\bf \ref{thm:12_some_}}) (see Section~\ref{sec:prelim}). These are  games in which there are no unplayable tokens.
    \item We make precise the notions of unplayable tokens and reduced positions (those without unplayable tokens at the end of the game) and establish an easy-to-check criterion that is necessary and sufficient for a position to be reduced: the {\bf NIRB condition (Theorem~\ref{thm:NIRB})} (see Section~\ref{sec:red}).
   \item  We establish the equivalence to a game that consists of only reduced positions  and modified moves that result in reduced positions after the move {\bf(Theorem~\ref{thm:GGcorr})}, and provide the necessary tools to play in the equivalent game (reduction of initial position and adjustments after a move) (see Section~\ref{sec:red}). 
    \item We determine the {\bf \P-positions} of \ruleset{Slow Exact $(n-1)$-Nim}, that is, \SAN~with $\bm{A=\{n-1\}}$ {\bf (Theorem~\ref{thm:P-pos_ex_k-1})} (see Section~\ref{sec:P-pos}), which complements the results in~\cite{GMMV2023} on the $M$-rule for this game. It may seem unnecessary to know an exact description of the \P-positions when a winning strategy is known. However, the \P-positions of this game not only exhibit a nice geometric structure, they also allow us to derive the \P-positions of another family of \ruleset{Slow$A$-Nim} games. 
    \item  We provide a tool for deriving the \P-positions of a game from the known \P-positions of a closely related game {\bf (Proposition~\ref{prop:GameExtended})} and use it to prove the {\bf \P-positions} of \SAN~with $\bm{A=\{n-1,n\}}$, where play is allowed on {at least $n-1$ stacks}  {\bf (Theorem~\ref{thm:P-pos_ex_AL})} (see Section~\ref{sec:P-posMG}). This game is  an instance of a family of games that form a counterpart to \ruleset{Slow Moore's $k$-Nim}, where play is allowed on at most $k$ stacks. These complimentary games have not previously been studied, and we have a conjecture for the \P-positions of \SAN~with $A=\{k,\dots,n\}$ {\bf (Conjecture~\ref{con:atleastk})}.
\end{itemize}
We start by introducing our notation in the next section  and conclude with open questions and a conjecture in Section~\ref{sec:open}.

\section{Notation and First Results}\label{sec:prelim}

For ease of readability, we will denote the game \ruleset{Slow$A$-Nim} by $\SN{n}{A}$. If the set $A$ is a singleton, we will use  $\SN{n}{i}$ instead of $\SN{n}{\{i\}}$. Throughout,  $\p=(p_1,\dots,p_n)$ denotes a position in the game, where $p_i$ is the height of the $i^{\text{th}}$ stack. Since the order of the stacks does not matter, we assume that any position is in \emph {canonical form}, that is, ordered with non-decreasing stack heights, so $p_1=\min(\p)$ and $p_n=\max(\p)$. An \emph{option} of a position $\p$ is a position that can be reached by a legal move from $\p$. We will denote it by $\p'=(p_1',\dots,p_n')$. We also use $\sp$ to denote the sum of the stack heights of position $\p$, and $\P_{n,k}$ to denote the $\P$-positions of the game \SN{n}{k}. For a more general set $A$, we denote the \P-positions by $\P_{n,A}$ and let $a=\min(A)$.

The games \SN{n}{A} fall into two categories, namely games for which $1 \in A$ and those where $1 \notin A$. Games in the first category will never have any tokens left at the end of the game, and $(0,\dots,0)$ is the only terminal position. By  contrast, games in the second category have multiple terminal positions - any position with fewer than $a=\min(A)$ non-zero stacks is terminal. These games will be the focus of Section~\ref{sec:red}. 

We first give a partial structure result on the characteristics of the \P-positions of games \SN{n}{A} that applies to both categories of \SN{n}{A} games. While only a partial result, Proposition~\ref{prop:Game_n_A_Strong} gives an indication that the number of odd stacks will play an important  role in the descriptions of the \P-positions of \SN{n}{A} games.

\begin{proposition} (Partial Structure Result) \label{prop:Game_n_A_Strong}
Let $\N$ and $\P$ be the set of \N- and \P-positions of the game \SN{n}{A} and let $o(\p)$ be the number of odd stacks of a position $\p$. 
\begin{enumerate}
    \item \label{it:Game_n_A_Strong1} If $o(\p)=0$, then $\p\in \P$;
    \item \label{it:Game_n_A_Strong2} If $o(\p) \in A$, then $\p\in \N$.
\end{enumerate}
\end{proposition}

\begin{proof}
We proceed by induction on the total number of tokens. If $\sp = 0$, then $\p$ is the terminal position, so the base case is proved. Now assume the statements are true for any $\p'$ with $\spp <\sp$. First assume that $o(\p)=0$.  Then  play on  $a\in A$ stacks leads to a position $\p'$ with $o(\p')=a \in A$, so $\p' \in \N$. On the other hand, if $\p \in\N$ with $o(\p)=a$, then playing on the $a$ odd stacks will lead to $\p'$ with $o(\p')=0$, so $\p' \in \P$. This completes the proof.
\end{proof}

 We now present results on \SN{n}{A} games with $A=\{1,k\}$. Theorem~\ref{thm:1k} gives results when $k$ is odd (for any $n$) or $k=n$ is even, while Theorem~\ref{thm:12_some_} gives results on the specific games when $A=\{1,2\}$ and $n=3,\dots,5$.  So far, no general pattern for the family $\SN{n}{\{1,2\}}$ has emerged.

\begin{theorem}\label{thm:1k} For the game \SN{n}{A} with $A=\{1,k\}$ one has:
\begin{enumerate} 
\item When $k$ is odd, then $\P_{n,A}=\{ \p \mid \sp \textrm{ is even}\}$, for any $n$.
\item When $k=n$ is even, then $\P_{n,A}=\{ \p \mid \textrm{both } \sp \textrm{ and the minimal stack } p_1  \textrm{are even}\}$.
\end{enumerate}
\end{theorem}

\begin{proof} \hfill
\begin{enumerate} 
\item If $k$ is odd, then each move, whether play is on one or on $k$ stacks (when possible),  changes the parity of $\sp$. Since the terminal position $(0,\ldots, 0)$ has the even sum  $\sp=0$, we have proved the claim.
\item Suppose now that $k=n$ is even. If $\p \in \P_{n,A}$, then playing on one stack will change the parity of the sum. If play is on all stacks, then it changes the parity of the minimal element $p_1$, so in both cases, $\p' \notin \P_{n,A}$.
Suppose now that $\p \notin \P_{n,A}$. We want to show that there is a move to a position $\p' \in \P_{n,A}$. Suppose that $\sp$ is even, and hence the minimal stack must be odd. This implies that one can play on all stacks, which changes the parity of the minimal stack, but not the parity of the sum, as $n$ is even, thus  $\p' \in \P_{n,A}$. If both  $\sp$ and $p_1$ are odd, playing on a minimal stack will change the parity of both the minimal stack and the sum. If  $\sp$ is odd and $p_1$ is even, then playing  on any stack that is not minimal will change only the parity of the sum. In both cases,  $\p' \in \P_{n,A}$. The only problematic case would be that all stacks are minimal and odd, but this cannot happen because $n$ is even, so $\sp$ would be even. \qedhere
\end{enumerate}
\end{proof}

\begin{theorem} \label{thm:12_some_}
 Let $o$ and $e$ denote whether an individual stack is odd or even and assume that $\p$ is in canonical form with non-decreasing stack heights. For the game \SN{n}{A} with $A=\{1,2\}$ we have
\begin{enumerate} 
\item $\P_{3,A}=\{\p \mid \p=(e,e,e) \textnormal{ or } \p=(o,o,o)\}.$
\item $\P_{4,A}=\{\p \mid \p=(e,e,e,e) \textnormal{ or } \p=(e,o,o,o)\}.$
\item $\P_{5,A}=\{\p \mid \p=(e,e,e,e,e) \textnormal{ or } \p=(e,e,o,o,o) \textnormal{ or }  \p=(o,o,e,e,o) \textnormal{ or } \p=(o,o,o,o,e)\}.$
\end{enumerate}   
\end{theorem}

\begin{proof} Note that the results for $n=3$ follow from those of $n=5$ with $p_1=p_2=0$ and those for $n=4$ follow by setting $p_1=0$. We now prove the result for $n=5$. Let $S=\P_{5,A}$. First we show there is no move from $S$ to $S$. We know from Proposition~\ref{prop:Game_n_A_Strong} that positions $\p$ with $o(\p)=0$ are in \P~and those with $o(\p)\in\{1,2\}$ are in $\N$, so we just need to consider positions $\p\in S$ with $o(\p)=3$ and $o(\p)=4$. 
\begin{itemize}
    \item $o(\p)=3$: If play is on one or two odd stacks, then $o(\p') \in \{1,2\}$. If play is on one even stack, then $o(\p')=4$ and  $p_5'=p_5$ is odd. If play is on two even stacks, then $o(\p')=5$. In all of these cases, $\p' \notin S$. If play is on one odd and one even stack, then $o(\p')=3$. If $\p=(e,e,o,o,o)$, then exactly two of $p_3', p_4'$ and $p_5'$ are odd. If $\p=(o,o,e,e,o)$, then either $p_5'$ is even or only one of $p_1'$ and $p_2'$ is odd. In either case, $\p' \notin S$, so there is no move from $\p$ with $o(\p)=3$ to $\p' \in S$. 
    \item $o(\p)=4$: If play is on the even stack (and possibly a second, odd stack), then $o(\p')=5$ or $o(\p')=4$ and $p_5'$ is odd. If play is on a single odd stack, then $o(\p')=3$ and $p_5'$ is even. Finally, if play is on two odd stacks, then $o(\p')=2$. In all cases,  $\p' \notin S$.
\end{itemize}
 Therefore, there is no move from $S$ to $S$. Next we need to show that from any $\p \notin S$, we have a move to $\p' \in S$. By Proposition~\ref{prop:Game_n_A_Strong}, we only need to consider positions with $o(\p) \in \{3,4,5\}$. If $o(\p)=5$, then play on $p_1$ and $p_2$ yields $\p'=(e,e,o,o,o)$. If $o(\p)=4$, then $p_5$ is odd. If the even stack is non-empty, then  play on both it and on $p_5$ leads to $\p'=(o,o,o,o,e)$. If the even stack is empty (and thus, equal to $p_1$), then play on $p_2$ to  obtain $\p'=(e,e,o,o,o)$. Finally we consider $o(\p)=3$ and distinguish according to the parity of $p_5$. 
 \begin{itemize}
     \item $p_5$ is even: If the second even stack is non-zero, play on that stack to obtain $\p'=(o,o,o,o,e)$. If the second even stack is zero, play on $p_2$ and $p_5$ to obtain $\p'=(e,e,o,o,o)$.
     \item $p_5$ is odd: Then there are four possibilities for the other two odd and two even stacks to be distributed: $(o,e,o,e,o)$, $(e,o,e,o,o)$, $(o,e,e,o,o)$,  and $(e,o,o,e,o)$. In the first and second case, play on $p_2$ and $p_3$, and in the third and fourth case, play on $p_2$ and $p_4$ leads to either $(o,o,e,e,o)$ or $(e,e,o,o,o)$.
 \end{itemize}
 Thus, there is a move from $\p\notin S$ to $\p' \in S$, which shows that the set of \P-positions is given by $S=\P_{5,A}$.
\end{proof}

We now turn our attention to the case where $1 \notin A$ and we need to deal with unplayable tokens.

\section{Reduced Positions and Playable Game Graph}\label{sec:red}

\subsection{Motivation}\label{subsec:motredpos}

\begin{example} \label{ex:redpos}
Suppose you play the game \SN{3}{2}. Then the positions $\p=(1,1,100)$ and $\tilde{\p}=(1,1,2)$ have the same game outcome - they are both \N-positions as the game is over after a single move (by playing on the two stacks of height one). Moreover, they have the same game tree of possible moves and the respective positions are of the same types. The extra $98$ tokens in $\tilde{\p}$ will never be played in this game, no matter which moves are made. 
\end{example}

How can we capture this idea of unplayable tokens? Imagine we draw the complete game graph and look at the terminal positions. For each stack, the smallest amount of tokens remaining in any of the terminal positions are unplayable no matter what moves are made.  This leads to the following definition. 

\begin{definition} Let $\T(\p)=\{t^{(1)}(\p),\ldots,t^{(r)}(\p)\}$ be the set of terminal positions of $\p$. Then for a given stack~$i$, the number of unplayable tokens is given by 
$$u_i(\p)=\min_j\{t^{(j)}_i(\p)\}.$$ The vector $u(\p)=(u_1(\p),\ldots,u_n(\p))$ represents the tokens that have no impact on the game outcome, while $r(\p)=\p-u(\p)$ represents the tokens that could be played in a sequence of legal moves, and hence are relevant for the game outcomes. A position with $\p=r(\p)$ is called {\em reduced} and  $r(\p)$ is the {\em reduction of $\p$}.
\end{definition} 

 Note that $r(\p)$ is unique since $u(\p)$ is unique. Also note that  positions with a common reduction are equivalent in terms of game outcome since they have the same playable tokens.  If a position $\p$ is reduced, then $u(\p)= (0,\ldots,0)$, which means that for each stack $i$ there is a sequence of legal moves that reduces that stack to zero, but not necessarily any of the other stacks. Furthermore,  if a position $\p$ is reduced and there is a position $\q \neq \p$ with $r(\q)=\p$, then there are infinitely many positions that reduce to $\p$.  Because $\q$ is not reduced, there is at least one stack with an unplayable token. Any additional tokens on this stack  will also be unplayable, without changing the playable tokens. Thus, any position so ``created" from $\q$ will also reduce to $\p$, which means there are infinitely many positions that reduce to $\p$.

\begin{example} \label{ex:redpos2}
In the game \SN{3}{2}, positions $\p=(1,1,m)$ with $m \ge 2$ have terminal positions $(0,0,m)$ and $(0,0,m-2)$.  Thus, $u(\p)=(0,0,m-2)$ and $r(\p)=\p-u(\p)=(1,1,2)$, making $(1,1,2)$ the only reduced position among positions $\p=(1,1,m)$ with $m \ge 2$. Having the common reduction $(1,1,2)$, these positions are equivalent with regard to game outcome because they have the same number of playable tokens in respective stacks.
\end{example}

 Note that $r(\p)$ has an important property that will allow us to reduce the game \SN{n}{A} to an equivalent game played on reduced positions. 

\begin{lemma} \label{lem:iso}
If $\p=(p_1,\ldots, p_n)$ is a position of \SN{n}{A} and $r(\p)$ is its reduction, then all sequences of moves that can be performed from $\p$ can also be performed from $r(\p)$,
and vice-versa.   
\end{lemma}

\begin{proof} A given sequence of moves from a position fails to be legal only if one of the moves requires subtraction of a token when the corresponding stack is empty.   Since by definition  $r_i(\p) \le p_i$ for any $i$,  any legal sequence from $r(\p)$ is legal from $\p$ also. Now assume  that there is a sequence of moves ending in a terminal position $t^{(j)}(\p)$ that can be performed from $\p$.  Then the number of times play is on stack $i$ is given by $p_i-t^{(j)}_i(\p)\le p_i-u_i(\p)=r_i(\p)$ by definition of $u_i(\p)$. Therefore, there are always enough tokens for play, so the move sequence is legal from $r(\p)$  also. 
\end{proof}

\subsection{The NIRB Condition}\label{subsec:NIRB}

Given the potential of reduced positions to simplify the analysis of the \P-positions, we need an easy way to determine whether a position is reduced or not. Theorem~\ref{thm:NIRB}  provides a powerful and simple means to do so and to determine $r(\p)$ when $\p$ is not reduced.  It is a core ingredient in the proofs of the main theorems.

\begin{theorem} \label{thm:NIRB}
Let $n \ge 3$ and $a=\min(A)$. A position $\p$ in game \SN{n}{A} is reduced if and only if it satisfies the NIRB (Nobody Is Really Big) condition\footnote{`Nobody' refers to `no stack', which would be more precise language, but the resulting abbreviation NSIRB does not roll off the tongue as easily.}, that is, 
\begin{equation*} \label{eq:NIRBcond}
p_n \leq \frac{\sp}{a}.
\end{equation*}
 We call $\left\lfloor \frac{\sp}{a}\right\rfloor$ the NIRB value of $\p$.
\end{theorem}

\begin{proof} 
We first prove the result when $A=\{a\}$ is a singleton.
Assume that the position $\p$ is reduced. Therefore, for each $i$, there is a sequence of $n_i$ moves from $\p$ to a position $\q$ with $q_i=0$. Since necessarily $n_i \ge p_i$ and each of these  $n_i$ moves decreases the number of tokens by $a$,  we have that $\sp \geq n_i\cdot a \geq  p_i\cdot a$ for all $i$, so $\p$ satisfies the NIRB condition. 

Next, we assume that the NIRB condition 
holds and show that the position is reduced.  We proceed by induction on the number of tokens \sp.  For $\sp<a$, we have that $p_n=0$ by NIRB, and thus $\p=(0,\dots,0)$, which clearly is reduced. 
Now assume that $\sp\ge a$, that $\p$ satisfies NIRB, and that any position $\p'$ with $\spp<\sp$ that satisfies NIRB is reduced.  Let $\alpha$ be the number of stacks of $\p$ that have maximal height and suppose $\alpha \leq a$. Define $m_0$ to be the move that plays on the $\alpha$ maximal stacks and on any $a -\alpha$ non-zero stacks and results in position  $ \p'$. Then $\spp=\sp-a$  and $\max(\p')=p_n-1.$ Thus, using that $\p$ satisfies NIRB, we have
$$a\cdot \max(\p') =a\cdot(p_n-1)\leq \sp-a = \spp,$$
so $\p'$ satisfies NIRB.  Thus, for every $i$ there exists a sequence of moves that leads from $\p'$ to a position $\q^{(i)}=(q_1,\dots, q_n)$ with $ q_i=0$ and  $q_j \ge 0$ for all $j$. Pre-pending the move $m_0$ to this sequence creates a move sequence from $\p$ to $\q^{(i)}$, hence $\p$ is reduced. 
Left to consider is the case where $\alpha > a$, so $\spp = \sp-a$  and $\max(\p')=p_n$.  If $\p'$ satisfies NIRB, then $\p$ is reduced as in the case where $\alpha \leq a$. If 
 $\p'$ does not satisfy NIRB, then
$ a \cdot \max(\p') > \spp=\sp-a$,  or equivalently, $a \cdot p_n +a > \sp$.
However, because there are $\alpha >a$ maximal stacks, 
$\sp \ge \alpha \cdot p_n\ge (a+1) p_n=a \cdot p_n+a,$
a contradiction. Thus, a position that satisfies NIRB is reduced when $A$ is a singleton. 

We now consider the case when $A$ is not a singleton. We need to show that $\p$ is reduced in game \SN{n}{A} if and only if $\p$ is reduced in \SN{n}{a}. If $\p$ is reduced in \SN{n}{a}, then it must be reduced in \SN{n}{A} also since $\{a\} \subset A$.  On the other hand, if $\p$ is reduced in \SN{n}{A}, then the move sequence  which leads to a position $\q^{(i)}$ with  $q_i=0$ and $q_j\ge 0$  may include moves that remove tokens from more than $a$ stacks. However, each such move can be replaced by a move that plays on stack $i$ and on any selection of $a-1$ of the other stacks that were played on in that move, leading to a position $\q'$. Stack $i$ will still be reduced to zero with these adjusted moves, and all other stack heights will satisfy $q_j'\ge q_j\ge 0$, so the adjusted move sequence is legal. Thus, $\p$ is  reduced in \SN{n}{a}, which completes the proof.
\end{proof}

\subsection{Associated Playable (Reduced) Game}\label{subsec:Playable}

Since unplayable tokens do not affect the game outcome, we want to remove them from consideration. In Example~\ref{ex:ghost} we illustrate how to arrive at an equivalent reduced game graph that eliminates the unplayable tokens. Note that playable tokens can become unplayable during game play but not vice versa:  for any option $\p'$ of $\p$ we have  $\T(\p')\subset \T(\p)$, thus  $u_i(\p')\ge u_i(\p) $ for all $i$.

\begin{example} \label{ex:ghost}
Consider the position $\p=(1,2,5,6)=(a,b,c,d)$ in game \SN{4}{3}, which is not reduced as $p_n\cdot 3 = 18 \nleq 14=\sp$; its reduction is $r(\p)=(1,2,3,3)$. On the left of Figure~\ref{fig:game_eq} we show the game tree $\mathcal{G}$ of $\p$, where labels on the edges identify the stack that is not played on. Playable tokens are shown in black and unplayable ones are in light gray. Newly  unplayable tokens are shown in blue (dark gray) when they become unplayable, then change to light gray in subsequent options. The playable game tree $\tilde{\mathcal{G}}$ (on the right) is obtained by replacing each position by its reduction. The allowed moves in $\tilde{\mathcal{G}}$ are the legal moves of \SN{n}{A} followed by reduction if needed. Reduction is needed whenever newly unplayable tokens occur as a result of a move. Moves that require the adjustment are indicated with a tilde in the playable game graph.  
\end{example}

\begin{figure}[!hbt]
    \centering
    \includegraphics[width=.9\linewidth]{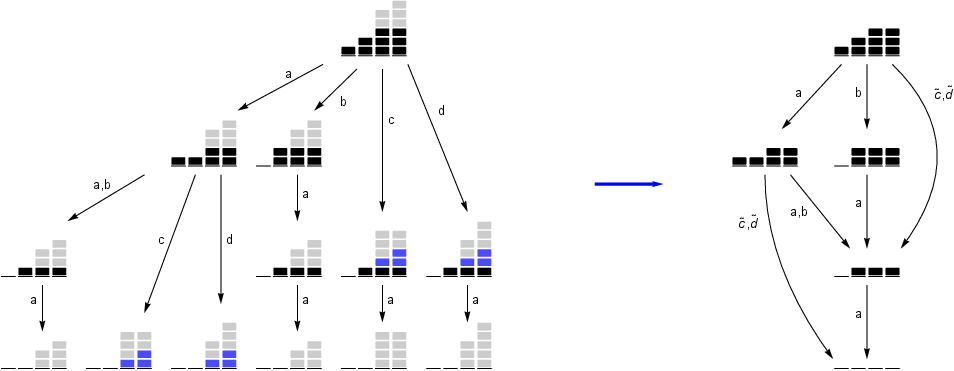}
    \caption{Mapping of game tree $\mathcal{G}$ of $\p=(1,2,5,6)$  (on the left) to its playable (reduced) game graph $\tilde{\mathcal{G}}$ (on the right). In $\mathcal{G}$, tokens corresponding to $r(\p)$ are shown in black, and unplayable tokens are shown in light gray. Newly unplayable tokens are shown in blue (dark gray) in the position where they become unplayable, and in light gray thereafter. In $\tilde{\mathcal{G}}$, moves of $\SN{n}{A}$ are replaced by the respective moves followed by reduction (if necessary). Moves where reduction of $\p'$ is needed due to newly unplayable tokens are indicated with a tilde. }
    \label{fig:game_eq}
\end{figure}

\begin{definition}
    The \emph{playable} or \emph{reduced game graph} $\tilde{\mathcal{G}}$ of game \SN{n}{A} consists of the reduced positions $r(\p)$ and adjusted moves $\tilde{\mathbf{m}}:=\mathbf{m}+\Delta u(\p')$, where $\Delta u(\p')=r(\p)-\mathbf{m}-r(\p')$ accounts for the newly unplayable tokens (if any) of $\p'$. We will denote the \emph{reduced game} played on $\tilde{\mathcal{G}}$ by \SNr{n}{A}. 
\end{definition}

By Lemma~\ref{lem:iso}, the possible move sequences from $\p$ and $r(\p)$ are identical, that is, the respective game trees from each of these positions are isomorphic. Thus, the outcome class of $\p$ is that of $r(\p)$ so we can use the playable game graph to determine the outcome class of any position.  This gives us the following theorem.

\begin{theorem} \label{thm:GGcorr} (Game Graph Correspondence) Let $\tilde{\mathcal{G}}$ be the playable game graph of \SN{n}{A} and let $\tilde{\P}$ and $\tilde{\N}$ be the sets of \P- and \N-positions, respectively, of the game defined by $\tilde{\mathcal{G}}$. For a position $\p$ of \SN{n}{A},   $\p \in \P_{n,A}$ only if $r(\p) \in \tilde{\P}$ and $\p \in \P_{n,A}^c$ only if $r(\p) \in \tilde{\N}$.
    
\end{theorem}

\begin{corollary}\label{cor:same_red}
    If \p~and \q~have the same reduction, that is, $r(\p)=r(\q)$, then  \p~and \q~are in the same outcome class.
\end{corollary}

Now that we have established that we can determine the outcome class of a position $\p$ in \SN{n}{A} from the outcome class of its reduction $r(\p)$, we will analyze the game \SNr{n}{A}. It is important to note this fact because the characteristics of the \P-positions are valid for the reduced position $r(\p)$, whose characteristics, such as parity of total number of tokens, number of odd stacks, etc., may differ from those of the unreduced position $\p$. To utilize the playable game graph $\tilde{\mathcal{G}}$ and to analyze the game $\SNr{n}{A}$, we need to 
\begin{enumerate}
    \item check whether the initial position is reduced, and if not, find its reduction; 
    \item determine which moves need adjustment and find the form of $r(\p')$. 
\end{enumerate}

\subsection{Initial Reduction}\label{subsec:init_red}
 The most obvious way to find the reduction of an initial  position $\p$ is to check whether $\p$ satisfies the NIRB condition, and if not, to iteratively decrease the maximal value. Such an algorithm will have order $O(\sp)$ because just one token may be eliminated in each iteration. We will discuss a more efficient algorithm that starts from the general structure of the reduced position and then performs a sequence of checks to determine the maximum stack height of the reduced position. An important ingredient is the function ``chop the top" $ctt(\p,m)$, which is defined as $ctt(\p,m)_i=\min(p_i,m)$. Note that  $\Sigma(ctt(\p,m))\le \sp$ for any $m$. 
 Before we present the algorithm, we establish some properties of $ctt(\p,m)$ that will ensure the algorithm arrives at the reduced position.

\begin{lemma} \label{lem:ctt_red} In the game \SN{n}{A}, if $\p$ is reduced, then $ctt(\p,m)$ is reduced for any value of $m$.
\end{lemma}

\begin{proof}
If $m \ge p_n=\max(\p)$, then $ctt(\p,m)=\p$, and nothing is to be proved. If, on the other hand, $m < p_n$, then there must be  $n_1\ge 0$ stacks that are of height at most $m$ and $n_2\ge 1$  stacks that have more than $m$ tokens, with $n_1+n_2=n$. 
Let $\q=ctt(\p,m)$. Because $\p$ is reduced, we have 
\begin{equation}\label{eq:pNIRB}
    a \cdot p_n \leq \Sigma(\p)=\sum_{i=1}^{n_1}p_i+\sum_{i=1}^{n_2}p_{(n_1+i)}.
\end{equation}
 To show that $\q$ is reduced we need that 
    $a\cdot m \leq \Sigma(\q)=\sum_{i=1}^{n_1}p_i+ n_2 \cdot m$.
If $a \le n_2$, then $a\cdot m \le n_2 \cdot m \le \sum_{i=1}^{n_1}p_i+ n_2 \cdot m=\Sigma(\q).$
If $a >n_2$, then 
\[a(p_n-m)>n_2(p_n-m)\ge \sum_{i=1}^{n_2}(p_{(n_1+i)}-m)=\sum_{i=1}^{n_2}p_{(n_1+i)} -n_2 \cdot m,\] 
which, using Equation~\eqref{eq:pNIRB} implies that
\[a \cdot m \leq a \cdot p_n -\sum_{i=1}^{n_2}p_{(n_1+i)} +n_2 \cdot m \leq \sum_{i=1}^{n_1}p_i+n_2 \cdot m= \Sigma(\p),\]
so $\q$ is reduced. 
\end{proof}

\begin{lemma}\label{lem:redx=cttxm}
    In the game \SN{n}{A}, for any $\p$, $r(\p)=ctt(\p,m)$, where $m$ is the largest integer $ m\le p_n$ for which $ctt(\p,m)$ is reduced.
\end{lemma}

\begin{proof}
    By Lemma~\ref{lem:ctt_red}, if $ctt(\p,m)$ is reduced for $m$, then $ctt(\p,m')$ is reduced for $0 \le m' \le m$.  Since reduction removes the minimal number of tokens required to make a position reduced,  $r(\p)=ctt(\p,m)$, where $m$ denotes the largest integer for which $ctt(\p,m)$ is reduced.  
\end{proof}

We now derive the algorithm. By Lemma~\ref{lem:redx=cttxm}, we need to find the largest $m$ for which $ctt(\p,m)$ is reduced. If $m \in (p_i,p_{i+1}]$ for some $i=1,\ldots,n-1$, then $\q=ctt(\p,m)$ is reduced if $a\cdot m\leq \Sigma(\q)=\sum_{j=1}^{i}p_j+(n-i)\cdot m$, or equivalently, 

\begin{equation}\label{eq:red-cond-M}
    (a+i-n)\,m \le \sum_{j=1}^{i}p_j=:\Sigma_i(\p).
\end{equation}
The largest value of $m$ that satisfies Equation~\eqref{eq:red-cond-M} is $\left\lfloor\frac{\Sigma_i(\p)}{a+i-n}\right\rfloor$ for $a+i-n>0$, so let $m_i=\left\lfloor\frac{\Sigma_i(\p)}{a+i-n}\right\rfloor$. The algorithm first checks whether $\p$ is already reduced, and if not, iteratively checks whether $m_i\in (p_i,p_{i+1}]$ for $i=n-1, n-2,\ldots,$ until such an $m_i$ is found. Then $r(\p)=ctt(\p,m_i)$. Example~\ref{ex:comp-init-red-A2} illustrates the algorithm.

\begin{minipage}{.9\linewidth}
\begin{algorithm}[H]
\caption{Initial Reduction}
\label{alg:initred}
\begin{algorithmic}[1]
\Require $\p=(p_1,\dots,p_n)$ and $a$
\State Initialize an array of length $n+1$ named $s$
\State $s[0] \gets 0$
\For{$i=1,\dots,n$}
   \State $s[i] \gets s[i-1]+p_i$
\EndFor
\If{$p_n\cdot a \leq s[n]$}
    \State \Return $\p$
\EndIf 
\For{$j=1,\dots, a-1$}
   \State $m \gets \left\lfloor\frac{s[n-j]}{a-j}\right\rfloor $
   \If{$m \in (p_{n-j},p_{n-j+1}]$}
    \State Compute $ctt(\p,m)$
    \State \Return $ctt(\p,m)$
\EndIf   
\EndFor
\end{algorithmic}
\end{algorithm}
\end{minipage}
 \vspace{0.2in}

\vspace{0.1in}

\begin{example}\label{ex:comp-init-red-A2}
Let $\p=(12,20,33,52,79,112,155,17 0)$ and $a=5$. In Phase I, the algorithm computes the partial sums. In Phase II, it computes the potential max value $m=\left\lfloor\frac{s[n-j]}{a-j}\right\rfloor$ (which is not the NIRB value, except for $j=0$) and checks whether it is in the interval currently being considered.  Table~\ref{tab:Alg} shows the computations, where the case $j=0$ corresponds to the check whether the position is already reduced. The complexity of the algorithm  depends on the number of stacks $n$ in Phase I and on the value of $a \leq n$ in Phase II, so overall, the algorithm is of order $O(n)$, the number of stacks.

\renewcommand{\arraystretch}{1.75}
\begin{table}[!htb]
\centering
\begin{tabular}{|l|c|c|c|c|c|c|c|c|c|c|}
\hline
\multirow{2}{*}{Phase I}  & $i$             & 0  & 1  & 2   & 3   & 4   & 5   & 6   & 7   & 8   \\ 
\cline{2-11} 
 & $s[i]$  & 0    & 12   & 32  & 65       & 117 & 196 & 308 & 463 & 633 \\ \hline\hline
\multirow{4}{*}{Phase II} & $j$             & 0  & 1  & 2   & 3   &   &  &  &  & \\ \cline{2-11} 
 & $m=\left\lfloor\frac{s[8-j]}{5-j}\right\rfloor$    & 126   & 115         & 102 & 98   &  &  &  & &     \\ \cline{2-11} 
 & Interval   & (170,$\infty$) & (155,170] & (112,155]   & (79,112]  & &  & & &  \\ \cline{2-11} 
& $m$ in interval? & N  & N  & N  & Y  &  &     &     &     &     \\ 
\hline
\end{tabular}
 \caption{Steps in the algorithm to reduce the initial position.}
    \label{tab:Alg}
\end{table}
\end{example}

\subsection{Reduction after a Move} \label{subsec:afterMove}

In Example~\ref{ex:ghost}, the moves that required a reduction after the move were those in which a maximal stack was omitted. We will now show that this is true in general, but that this is not the only condition.

\begin{lemma} [General Reduction Lemma] \label{lem:redcondgenA} In \SN{n}{A},  a move on $\ell \in A$ stacks from a position \p~to a position $\p'$ only needs a reduction after the move if one of the following holds:
\begin{itemize}[leftmargin=4em]
\item[(R1)] play is on $\ell>a=\min(A)$ stacks, all maximal stacks are played on, and $\sp=a \cdot p_n +r$, where $0 \leq r < \ell-a$.
\item[(R2)] play is on $\ell$ stacks, at least one of the maximal stacks is omitted, and
 $\sp=a \cdot p_n +r$, where $0 \leq r < \ell$.
\end{itemize}
\end{lemma}

\begin{proof}
 We consider two cases, depending on whether all maximal stacks are played on or not. If play is on $\ell$ stacks and includes all maximal stacks, then $p_n'=p_n-1$ and reduction is needed only if $\sp-\ell=\spp<a\cdot(p_n-1)\le \sp-a$, which implies that $\sp= a \cdot p_n+r$, where $0 \leq r < \ell-a.$ Thus,  no reduction is needed when play is on the minimal number $a$ of stacks and includes all maximal stacks. 
In the second case, when at least one maximal stack is omitted, we have that $p_n'=p_n$ and reduction is needed only if $\sp-\ell<a\cdot p_n\le \sp$, which implies that $\sp= a \cdot p_n+r$, where $0 \leq r < \ell.$
  \end{proof}

We will apply Lemma~\ref{lem:redcondgenA} to the sets $A=\{n-1\}$ in Section~\ref{sec:P-pos} and $A=\{n-1,n\}$ in Section~\ref{sec:P-posMG}, respectively.

 \section{\P-Positions for \SN{n}{n-1}}\label{sec:P-pos}
 
The \P-positions of  \SN{n}{k} where $k=n-1$ are characterized by $o$, the number of stacks with odd height,  and by $s = \sp \bmod(2k)$.  For ease of readability, we  abuse notation and write $\p=(s,o)$ for a position $\p=(p_1,\ldots,p_n)$ of type $(s,o)$.

\begin{lemma} [Options when no reduction is needed after the move] \label{lem:noredopts} In the game \SN{n}{k} with $k=n-1$, when no reduction is needed after a move from a reduced position $\p=(s,o)$ to $\p'=(s',o')$, then  
\begin{equation*}\label{eq:noredopt_s}
s'= (s-k)\bmod (2k) =\left\{\begin{array}{cl}
s+k & \text{if  $s < k$} \\
s-k & \text{if  $s \ge k$} \\

  \end{array}\right.
  \end{equation*}
 and 
\begin{equation*}\label{eq:noredopt_o}
  o'=\left\{\begin{array}{cl}
      n-o+1 & \text{if the omitted stack is odd} \\
      n-o-1 & \text{if the omitted stack is even.}
  \end{array}\right.
  \end{equation*}
 For positions with $0<o<n$, both  options for $o'$ are legal, while positions with $o=n$ and $o=0$ have only the options  $o'=1$ and  $o'=n-1$, respectively. 
\end{lemma}  

\begin{proof}
When no reduction is needed, then the sum of stack heights is reduced by  $k$ tokens, so $s' = (s-k)\bmod(2k)$. For the values of $o'$, when the omitted stack is odd, then all the even stacks change parity (a total of $n-o$ stacks) and the omitted stack contributes another odd stack for $\p'$. When the omitted stack is even, then $n-o-1$ even stacks change parity. 
 \end{proof}

On the other hand, when reduction is needed, then the reduced option and its characterization are given by Lemma~\ref{lem:redposodd}. 

\begin{lemma} [Options when reduction is needed after amove] \label{lem:redposodd} 
 Let $ \alpha $ be the number of  maximal stacks of $\p$. In the game \SN{n}{n-1}, when reduction is needed after a move from a reduced position $\p=(s,o)$, then the reduction of the option $\p'$ is given by 
  $$ r(\p')= \left\{ \begin{array}{ll}
      (p_1-1,\ldots,p_n-1) & \text{ if } s\not \equiv 0\bmod k \\
      (p_1-1,\ldots,p_{n-\alpha}-1,p_n-2,\ldots,p_n-2) & \text{ if } s\equiv 0\bmod k. \\
  \end{array}\right.
  $$
  Furthermore, if $0\leq s<k$, then $\max(\p)$ is even, otherwise $\max(\p)$ is odd. The  move is from $\p=(s,o)$ to $r(\p')=(s',o')$, which is characterized  as follows:
  
\begin{itemize}
    \item  If $s\not \equiv 0\bmod k$, then $o'=n-o$ and $s'=\left\{\begin{array}{ll}
          s+k-1 & \text{ if }0< s < k \\
          s-(k+1) & \text{ if }k < s <2k. 
      \end{array} \right.$
  \item If $s\equiv 0\bmod k$, then $(s',o')=\left\{
       \begin{array}{ll}
          (k-1-\alpha, \, n-o-\alpha) & \text{ for $s=0$ and $\alpha \le \min\{n-o,n-2\}$} \\
          (2k-1-\alpha, \, n-o+\alpha) & \text{ for $s=k$ and $\alpha \le \min\{o,n-2\}$.} 
           \end{array} \right.$
  \end{itemize}
\end{lemma}

\begin{proof}
Applying Lemma~\ref{lem:redcondgenA} for $A=\{k\}$ with $k=n-1$, reduction after a move from  position $\p= (s,o)$ is needed only when a maximal stack is omitted and
\begin{equation}\label{eq:redcond}
\sp=p_n\cdot k + r, \text{ with } 0 \leq r < k.
\end{equation}

The adjustment to make the option reduced will require to lower the omitted maximal stack, and potentially the resulting maximal stack(s). 

First, we consider the case $s\not\equiv 0 \bmod k$, which implies that $r>0$.  Using NIRB and Equation~\eqref{eq:redcond}, we have that  $p_n \le \sp/k=p_n +r/k$. Since $p_n$ is an integer and $r \ge 1$, we also have that $p_n \le p_n + (r-1)/k$.  Thus, $\tilde{\p}=(p_1-1,\ldots,p_n-1)$, where the omitted maximal stack is reduced by one token, satisfies NIRB since
$$\tilde{p}_n=p_n-1\le p_n+\frac{r-1}{k}-1=\frac{p_n \cdot k +r-(k+1)}{k} = \frac{\Sigma(\tilde{\p})}{k},$$
so $\tilde{\p}=r(\p')$. Having removed $n=k+1$ tokens,  $s'=(\sp-(k+1) )\bmod(2k)$, thus $ s'=s+k-1$ for $0<s<k$ and $s'=s-(k+1)$ for $k<s<2k$. Since all stacks have changed parity, $o'=o(r(\p'))=n-o$.

 Now let $s \equiv 0 \bmod k$, that is, $\sp = p_n \cdot k$. First we prove that $\alpha \leq n-2$.  If $\alpha = n$ then $\sp = p_n  (k+1) \ne p_n \cdot k$, a contradiction.   If $\alpha = n-1$, then  $p_1=0$. Thus, all maximal stacks are being played on, a contradiction to the fact that reduction is needed, so $\alpha \le n-2$. Now assume that we omit one of the maximal stacks. Reducing just the omitted maximal stack by one (as in the case of $s \not \equiv 0 \bmod k$) is not enough, because $\tilde{\p}=(p_1-1,\ldots,p_n-1)$ does not satisfy NIRB:
$$\frac{\Sigma(\tilde{\p})}{k}=\frac{\sp-n}{k}=p_n-\frac{n}{k}<p_n-1=\tilde{p}_n.$$
Therefore, we adjust $\tilde{\p}$ by decreasing all of its $\alpha$ maxima  of value  $p_n-1$ by one and show that the resulting position $r(\p')=(p_1-1,\dots,p_{n-\alpha}-1,p_n-2,\dots, p_n-2)$
is reduced: 
$$ p_n-2 \leq \frac{\sp}{k}-2=\frac{\sp-2(n-1)}{k}=\frac{\sp-n-(n-2)}{k}
    \leq\frac{\sp-n-\alpha}{k}=\frac{\Sigma(r(\p'))}{k}.$$
The total number of tokens has been decreased by $n+\alpha$, so $\Sigma(r(\p'))=\sp-k-(\alpha+1)$. If $s=0$, then $s'=\Sigma(r(\p'))\bmod (2k)=k-(\alpha+1)$, while for $s=k$, we have $s'=2k-(\alpha+1)$.  

To determine the parity of the maximum, we use the definition of $s$ and the reduction criterion.  By definition of $s$, we have  $\sp=c\cdot (2k)+s=2c\cdot k+s$, while the reduction criterion states that $\sp=p_n\cdot k + r$, with $0 \leq r \le k-1$.  When $0\le s <k$, we have that $2c=p_n$, so the maximum is even, and  $r=0$ corresponds to $s=0$.  If $k\le s <2k$, then $\sp=c \cdot 2k+k+r=(2c+1)k+r$ with $0\le r <k$, so the maximum is odd and $r=0$ corresponds to $s=k$.

The values for $o'$ are determined as follows. When $s=0$, the maximum is even, thus $\alpha \leq n-o$. Decreasing each stack by 1 produces $n-o$ odd stacks, then decreasing the $\alpha$ (now odd) maxima by an additional token yields $o'=n-o-\alpha$. When $s=k$, the maximal stacks are odd, so $\alpha\leq o$, and $o'=n-o+\alpha$. 
 \end{proof}

Note that the assertion that the maximum is even or odd results from the reduction condition. There obviously can be positions with $s\ge k$ that have an even maximum. It just means that for such a position, reduction will not be needed. For example, the position $\p=(6,7,8,8,8)$ is reduced and has $s=37 \bmod(8)=5>4=k$ and an even maximum.  Even if we do not play on all maximal stacks, that is, make the move to $\p'=(5,6,7,7,8)$, we still have that $\p'$ satisfies NIRB and does not need reduction.

Lemmas~\ref{lem:noredopts} and~\ref{lem:redposodd} give the form of all possible options of a position $\p$ in the playable game graph $\tilde{\mathcal{G}}$ of \SN{n}{n-1}, incorporating the adjustment needed for (some of) the moves in $\tilde{\mathcal{G}}$.  We are now ready to state the result on the \P-positions of \SNr{n}{n-1}, \ruleset{Slow$A$-Nim} played on all but one stack in the playable game graph.

\begin{theorem} \label{thm:P-pos_ex_k-1} For the game \SNr{n}{k} with  $k=n-1$, let $s = \sp \bmod (2k)$ and  $o$ be the number of stack heights that are odd. Then for all positions of the game, $s$ and $o$  have the same parity, and the set of \P-positions is given by 

$$\begin{aligned}
  \P_{n,n-1} = & \{(s,o)\mid s,o \text{ have same parity},\, 0 \leq s < k-1, \, o\leq s\}\cup \{(n-2,o)\mid o \text{ has same parity as }n\} \\
  & \cup\{(s,o)\mid s,o \text{ have same parity}, \,k-1 < s < 2k-1, \, o \le 2(k-1)-s\}=:S_1\cup S_2 \cup S_3.
\end{aligned}$$
\end{theorem}

We can visualize the \P-positions via a {\it position grid} that shows which combinations of values of $s$ and $o$ are \P- or \N-positions. Figure~\ref{fig:SN(9,8)} depicts the \P-positions of \SN{9}{8}  with colored cells. Cells for which $s$ and $o$ have different parity are indicated with a dot. The white cells represent \N-positions.

\begin{figure}[!htb]
     \centering   
         \includegraphics[width=.3\textwidth]{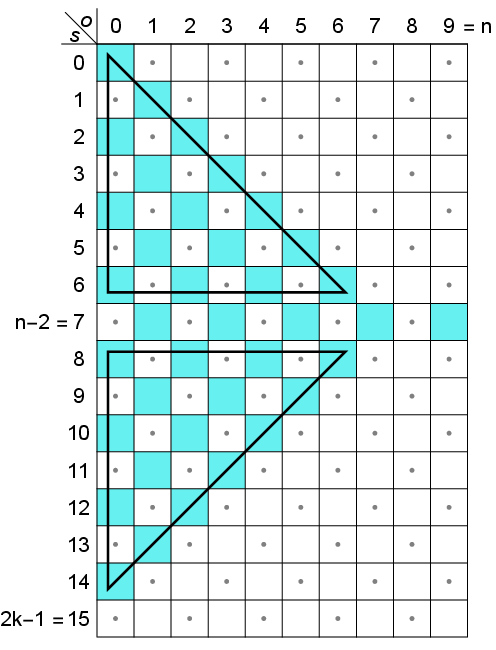}
        \caption{Position grid of the games \SNr{n}{k} for $k=n-1$, illustrated for $n=9$ and $k=8$. Cells representing \P-positions are colored, and combinations of values of $s$ and $o$ that cannot occur are shown with a dot. White cells represent \N-positions.}
        \label{fig:SN(9,8)}
\end{figure}

\begin{proof}
First we show that all positions $(s,o)$ have the same parity. If $o$ is odd, then $\sp$ is odd. The remainder of an odd value with regard to an even divisor is odd, so $s$ is odd. If $o$ is even, then $\sp$ is even, and consequently, the remainder is also even. Thus, $s$ and $o$ must have the same parity. 

Now we prove that $\P_{n,k}$ as given above is the set of \P-positions of this game. We first show that every move from a position in $\P_{n,k}$ without reduction is to a position in $\P_{n,k}^c$.  By Lemma~\ref{lem:noredopts}, $s'= (s-k)\bmod{(2k)}$ and $o'\in\{n-o-1,n-o+1\}$. For $\p \in S_1$,  we have $s'=s+k \geq k$ and $o'\ge n-o-1
\ge n-s-1=(k+1)-s-1 > 2(k-1)-(s+k)=2(k-1)-s'$, so $\p' \in P_{n,k}^c.$ Similarly, for $\p \in S_3$, we have $s' = s-k \leq k-2$ and $o' \ge  n-o-1 \ge n-(2(k-1)-s)-1=2+(s-k)>s'$, so $\p' \in \P_{n,k}^c$. Finally, $\p \in S_2$ with $(k-1,o)$ has option $\p'=( 2k-1, o')\in \P_{n,k}^c$.

We next turn to the moves that require a reduction. We use a graphical argument based on Lemma~\ref{lem:redposodd}. We color the positions in $\P_{n,k}$ under consideration in a darker cyan and depict their options as vertically striped cells. If $k < s \le 2k-2$, then the option of $\p=(s,o)$ is given by $(s',o')=(s-(k+1),n-o)$. Geometrically, the positions are moved ``up" in the position grid by $k+1$ rows, and then vertically reflected across the mid line as shown in Figure~\ref{fig:evenRedBZ}. (If $n$ is odd, then the mid line is the vertical line separating columns $(n-1)/2$ and $(n+1)/2$, while for even $n$, it is the vertical line through the center of the squares with $o=n/2$.) Similarly, if $0 < s < k$, then the option is $(s',o')=(s+k-1,n-o)$, that is, a shift ``downwards" by $k-1$ rows, and then the same type of reflection, shown in Figure~\ref{fig:evenRedLZ}. All the options are in $P_{n,k}^c$, except the option of $\p=(k-1,n)$, which maps to $\p'=(2k-2, 0) \in \P_{n,k}$. However, $\p$ has only odd stacks, and by Lemma~\ref{lem:redposodd}, the maximum is even when reduction is needed, so this move does not occur. 

\begin{figure}[!htb]
     \centering
     \begin{subfigure}[b]{0.31\textwidth}
         \centering
         \includegraphics[width=0.9\textwidth]{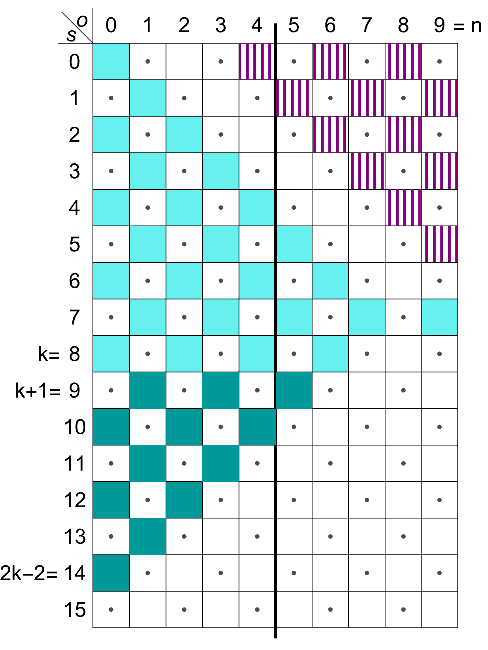}
         \caption{Case $s>k$.}
         \label{fig:evenRedBZ}
     \end{subfigure}
     \begin{subfigure}[b]{0.31\textwidth}
         \centering
         \includegraphics[width=0.9\textwidth]{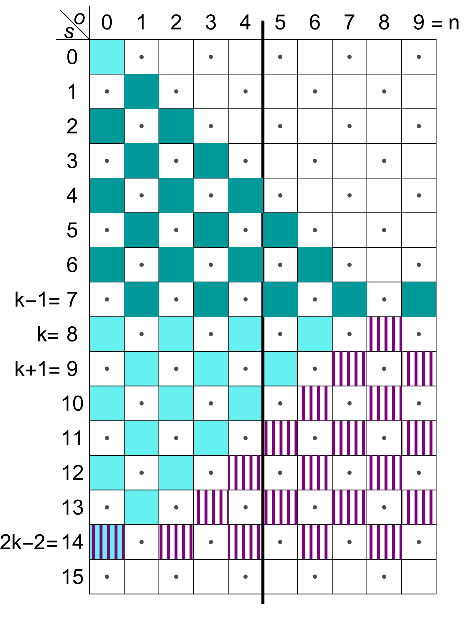}
         \caption{Case $0<s<k$.}
         \label{fig:evenRedLZ}
     \end{subfigure}
     \begin{subfigure}[b]{0.295
     \textwidth}
         \centering
         \includegraphics[width=1\textwidth]{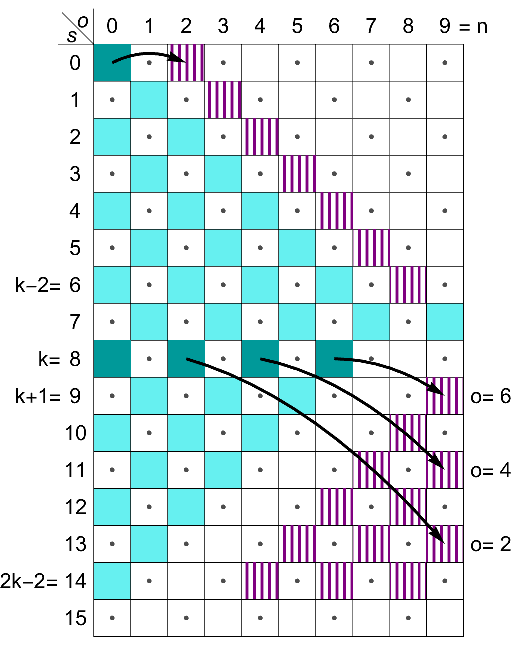}
         \caption{Cases $s=0$ and $s=k$.}
         \label{fig:evenRedZ}
     \end{subfigure}
        \caption{Visualization of the moves for $\p \in \P_{n,k}$ when reduction is needed, illustrated using the position grid for $n=9$ and $k=8$. Cells under consideration are colored in darker cyan, and their options are depicted as vertically striped cells. For (c), the arrow points to the top of a diagonal of options associated with the respective positions. }
        \label{fig:OPtsEvenRed}
        \hfill
\end{figure} 
 Left to consider are the positions with $s=0$ and $s=k$,  where the options depend on the number $\alpha$ of maxima, so we need to consider all possible values of $\alpha$. For $s=0$, there is just  $\p=(0,0)$ to be considered.  Since $o=0$,  we have $1 \leq \alpha \leq n-2=k-1$. The options are of the form $(k-1-\alpha,n-\alpha)$, or equivalently, $(i,i+2)$ for $i=0,\ldots,k-2$, illustrated with an arrow from $\p$ to the top of the descending diagonal in the upper part of the grid of Figure~\ref{fig:evenRedZ}.  Next we consider $s=k$, where the options are given by $(s',o')=(2k-1-\alpha, k+1-o+\alpha) $ with $\alpha \le \min\{o,n-2\}$. When $o=0$, no reduction is required. This is because by Lemma~\ref{lem:redposodd}, the maximum would have to be odd, but $\p$ has no odd stacks.   For positions $(k,o)$ with $o \ge 2$ and $\alpha \le o$, there are $o$ possible  options, and they form the diagonals increasing from position $(2k-2,n+1-o)$ (for $\alpha=1$) to position $(2k-1-o,n)$ (for $\alpha = o$) in the lower part of the grid of Figure~\ref{fig:OPtsEvenRed}(c). 
In all cases, the options are in $\P_{n,k}^c$, so there is no move from $\P_{n,k}$ to $\P_{n,k}$.

We next show that for any $\p \in \P_{n,k}^c$ we have a move to $\p' \in \P_{n,k}$.  We first consider whether a non-reduction move exists, and then deal with any cases in which such a move does not exist. By Lemma~\ref{lem:noredopts}, positions in $\P_{n,k}^c$ with $s <  k-1$ get mapped to $(s+k,o')$ and those with $s \geq k$ get mapped to $(s-k,o')$,  where $o' \in\{n-o-1,n-0+1\}$ in both cases.  In each case there is a shift ``up" or ``down" by $k$ rows, followed by a vertical reflection across the mid line. The options of $\p$ are located on the left and right of the reflected cell. Figure~\ref{fig:OPtsEvenNRW} shows the positions of  $\P_{n,k}^c$ that are under consideration in purple,  together with examples of positions in the interior and on the edge of $\P_{n,k}^c$ and their respective options (striped cells). For positions in the interior, both options are in $\P_{n,k}$, so the non-reduction move is guaranteed as we can leave out either an odd or an even non-maximal stack. Likewise,  $\p=(2k-1,1)$, while on the edge of $\P_{n,k}^c$, does have both non-reduction options available. Positions on the edge of $\P_{n,k}^c$ without guaranteed non-reduction move to $\P_{n,k}$ are marked with a diagonal in Figure~\ref{fig:OPtsEvenNRW}.

\begin{figure}[htb]
     \centering
     \begin{subfigure}[b]{0.35\textwidth}
         \centering
         \includegraphics[width=.9\textwidth]{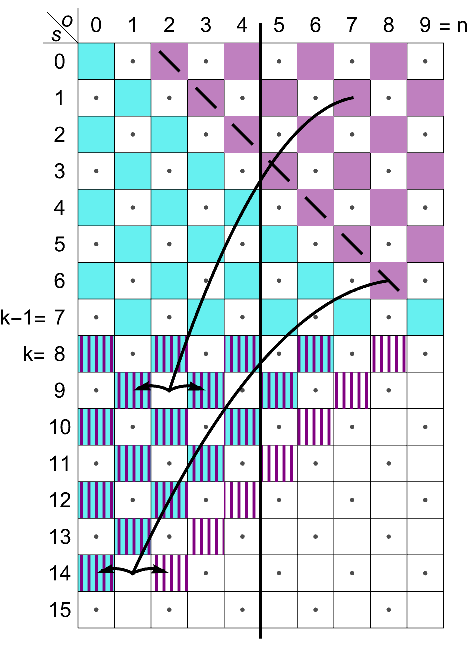}
         \caption{Options for $\p \in \P_{n,k}^c$ with $s<k$.}
         \label{fig:WUHNR}
     \end{subfigure}
     \hspace{0.1\textwidth}
     \begin{subfigure}[b]{0.35\textwidth}
         \centering
         \includegraphics[width=.9
\textwidth]{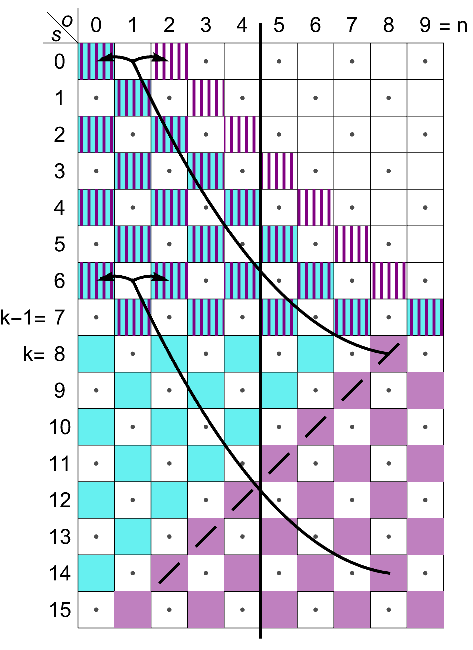}
         \caption{Options for $\p \in \P_{n,k}^c$ with $s\ge k$.}
         \label{fig:edge1}
     \end{subfigure}
        \caption{Visualization of the non-reduction moves of $\p \in \P_{n,k}^c$, illustrated for $n=9$ and $k=8$. Arrows indicate the two non-reduction options. Positions without a guaranteed non-reduction move are marked with a diagonal.}
        \label{fig:OPtsEvenNRW}
\end{figure} 

These positions are given by $\{(s,s+2)\mid0 \le s<k-1\}\cup\{(s,2k-s)\mid k \le s \le 2k-2\}$ and they have only the left option to $(s',n-o-1)\in \P_{n,k}$ available. This non-reduction move is possible when either the omitted stack is an even non-maximum or when omitting an even maximum does not require reduction. Thus,  reduction is needed if the position is of the form 
\begin{equation*}\label{eq:problemcases}
    \p=(\underbrace{p_1,p_2,\ldots, p_o}_{o \text{ odd stacks}},\underbrace{p_n,\ldots,p_n}_{\text{even stacks}})
\end{equation*}
where all even stacks are maximal. By Lemma~\ref{lem:redposodd}, $p_n$ is odd for $k \le s <2k$,  so  the non-reduction move to $(s',n-o-1)\in \P_{n,k}$ is available for $\p \in \{(s,2k-s)\mid k-1<s \le 2k-2\}$. 

Left to consider are the positions of the form $\{(s,s+2)\mid0 \le s<k-1\}$ which have an even maximum value. For $s \ge 1$, the reduction move is to $\p'=(s+k-1,n-o)=(s+k-1,n-(s+2))=(k-1+s,k-1-s)$. Geometrically, this is the diagonal of the form  $\{(2(k-1)-i,i)\mid i=1,\ldots,k-2\}\in S_3$. For $s=0$,   $\p=(0,2)$, that is, $\p$ has two odd stacks and $\alpha=n-2$ maximal stacks, so $\p'=(0,0)$ by Lemma~\ref{lem:redposodd}. This shows that for every $\p \in \P_{n,k}^c$ there is a move to $\p' \in \P_{n,k}$. Therefore, $\P_{n,k}$ is the set of $\P$-positions of \SNr{n}{n-1}.
 \end{proof}

\section{\P-Positions for \SN{n}{\{n-1,n\}}}\label{sec:P-posMG}

We now consider the game where we allow play on {\bf at least} $n-1$ stacks,  that is, the game \SN{n}{\{n-1,n\}}. This is the game \SN{n}{n-1} with one additional move, the move on all stacks. By Proposition~\ref{prop:Game_n_A_Strong}, the column in the position grid corresponding to $o=n$ now consists of \N-positions. In fact, this the only difference in the position grids of these two closely related games. To prove this hypothesis, we will use Proposition~\ref{prop:GameExtended}. It provides a general tool for proving the pattern of the \P-positions of a new game that differs from a solved game by only a few moves. 

\begin{proposition} \label{prop:GameExtended}
Let $\N$ and $\P$ be the set of \N- and \P-positions, respectively,  of a game $G$ with move set $M.$ Assume that we extend $G$ to $G'$ by allowing additional moves given by the move set $E$, that is, $G'$ has move set $M'=M \cup E$ with $M \cap E =\emptyset$. 
Suppose there exists a set $Q\subset \P$ such that:
\begin{enumerate}
    \item \label{it:GameExtI1} No move in $E$ links two elements of $\P \setminus Q$.
    \item \label{it:GameExtI2} For any $\p\in \N$ with a move in $M$ to a position in $Q$, there exists a move in $M'$ to a position in $\P\setminus Q$.
    \item \label{it:GameExtI3} From every position $\p\in Q$, there is a move in $E$ to a position in $\P \setminus Q$.
\end{enumerate}
Then, $\N'=\N \cup Q$ and $\P'=\P \setminus Q$ are the sets of \N- and \P-positions of the game $G'$.
\end{proposition}

\begin{proof}
We show that the restrictions on the set $Q$ ensure that the properties that determine the set of \P-positions of a game  hold for $A=\P\setminus Q$ and $B=\N \cup Q$.  Condition \ref{it:GameExtI1}
 ensures that there is no move from $A$ to $A$ using moves in $E$, and by the definition of $\P$, no move in $M$ links any positions in $A \subset \P$. Now suppose $\p \in B=\N \cup Q$. We need to show that there is a move from $\p$ to $A$.
If $\p\in \N$, then by definition of $\N$ there is move from $\p$ to $\p' \in \P$. If this move is to a position in $Q \subset \P$, then condition \ref{it:GameExtI2} ensures there is an alternate move in $M'$ to a position in $A$.  If $\p\in Q$, then there is a move from $\p$ to $A$ by condition \ref{it:GameExtI3}. Therefore, $\P'=A$, $\N'=B$ and the proof is complete.   
\end{proof}

We now state the result for the game \SNr{n}{\{n-1,n\}}. Figure~\ref{fig:SN(9,8)ext} compares the position grids of \SNr{n}{n-1} (on the left) and \SNr{n}{\{n-1,n\}} (on the right), illustrated for $n=9$. They indeed only differ  in position $\p=(n-2,n)$, which changed from being a \P-position to being an \N-position, as required by Proposition~\ref{prop:Game_n_A_Strong}.

\begin{theorem} \label{thm:P-pos_ex_AL}
For the game $\SNr{n}{A}$ with  $A=\{n-1,n\}$, let $s = \sp \bmod (2n-2)$ and  $o$ be the number of stack heights that are odd. Then for all positions in the game, $s$ and $o$  have the same parity, and the set of \P-positions is given by 
$$\begin{aligned}
  \P_{n,A} = & \{(s,o)\mid s,o \text{ have same parity}, \,  s < n-2, \, o\leq s\}\cup \{(n-2,o)\mid o,n \text{ have same parity}, \, o \neq n\},    \\
  & \cup\{(s,o)\mid s,o \text{ have same parity}, \, n-2 < s < 2n-3,  \,o \le 2(n-2)-s\}=:S_1\cup S_2 \cup S_3.
\end{aligned}$$
\end{theorem}

\begin{figure}[!htb]
     \centering
     \begin{subfigure}[b]{0.40\textwidth}
         \centering
  \includegraphics[width=.79\textwidth]{P-posSN98.eps}
         \caption{\P-positions of \SNr{9}{8}.}
         \label{fig:P-pos-odd}
     \end{subfigure}
     \hspace{0.1\textwidth}
     \begin{subfigure}[b]{0.40\textwidth}
         \centering
         \includegraphics[width=.79\textwidth]{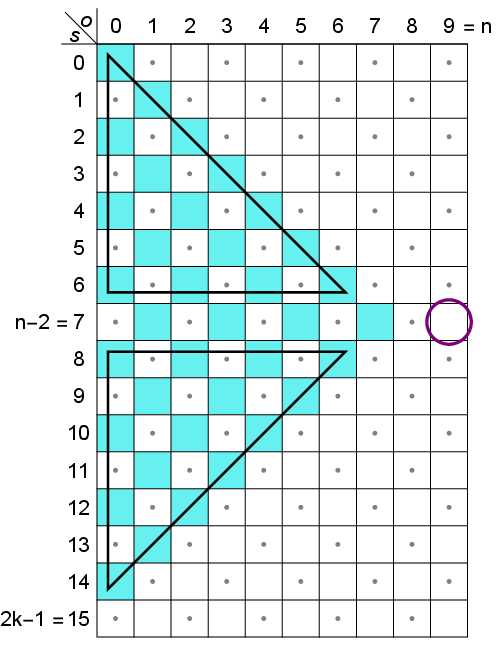}
         \caption{\P-positions of \SNr{9}{\{8,9\}}.}
         \label{fig:edge}
     \end{subfigure}
        \caption{Examples of \P-positions of games \SNr{n}{A} for $A=\{n-1\}$ on the left  and $A=\{n-1,n\}$ on the right, illustrated for $n=9$.  The circle indicates that position $(n-2,n)$ has changed from being a \P-position of \SNr{9}{8} to being an \N-position of \SNr{9}{\{8,9\}} by allowing the added move on all stacks. }
        \label{fig:SN(9,8)ext}
\end{figure}

\begin{proof} We will use Proposition~\ref{prop:GameExtended}. The set $\P$ equals the set $\P_{n,n-1}$ of Theorem~\ref{thm:P-pos_ex_k-1} and $s=\sp \bmod (2(n-1))$. The only added move is the move on all stacks, which we will refer to as POAS. Therefore, $E = \{(1,\ldots,1)\}$, and since $\p= (n-2,n)$ is the only position that changes its type as a result of POAS, we have that $Q = \{(n-2,n)\}$.  The non-reduction POAS move removes $n$ tokens and changes the parity of every stack. Applying Lemma~\ref{lem:redcondgenA} for  $a=n-1$, reduction is needed for POAS ($\ell=n)$ if and only if $\sp = (n-1)p_n$, the same condition as for the move on all but one stack in the case  $s\equiv 0\mod(n-1)$. Thus, the reduced option for POAS is  $r(\p')=(p_1-1,\dots,p_{n-\alpha}-1,p_n-2,\ldots,p_n-2)$, which means that some of the conditions of Proposition~\ref{prop:GameExtended} have already been proved for the game $\SN{n}{n-1}$. We check each of the three conditions:

\begin{enumerate} 
\item[(1)] Let $\p=(s,o)\in\P\setminus Q$.   The non-reduction POAS move results in $\p'= (s',n-o)$ where $s'=s+n-2$ if $s<n$ and $s'=s-n$ if $s \ge n$. Thus, the options are obtained by either a shift down in the grid by $n-2$ or a shift up in the grid by $n$ and then a reflection across the mid line. It can be easily verified by inspection in Figure~\ref{fig:SN(9,8)ext}(b) that the options are in $(\P\setminus Q)^c$. Second,  when a reduction is needed, the reduced position  is of the same form as when play is on all but one stack, and therefore $r(\p')\in \P^c$. Thus, for $\p \in \P\setminus Q$ we have that $\p' \in (\P\setminus Q)^c$.

\item[(2)]  We now consider positions $\p \in \N$ from which there is a move on $k=n-1$ stacks to  $\p'=(s',o')=(n-2,n)$. By Lemma~\ref{lem:noredopts}, the non-reduction   move is from   $\p = (2n-3,1)$ by playing on all but the single odd stack. Instead, we omit one of the even stacks and move to $\p''=(n-2,n-2) \in \P\setminus Q$. Reduction may only be needed when we omit a maximal stack, that is when $\p=(p_1,p_n,\ldots,p_n)$ with $p_1$ odd and $\sp=(n-1)p_n+p_1$.  If a reduction move would be needed from this position, then   Lemma~\ref{lem:redposodd} indicates that for $s=2n-3$, the maximum is odd,  a contradiction, so the move to $\p''$ is available. 

On the other hand, the position $\p'=(s',o')=(n-2,n)=(k-1,n)$ cannot be the result of a reduction move  by Lemma~\ref{lem:redposodd}.  If $s\not \equiv 0\bmod k$, then  $\p'=(s',o')=(k-1,n)$ would imply that $s=0$ or $s=2k$, a contradiction. If $s \equiv 0\bmod k$, then $s=0$ or $s=k$. In the first case,  $(s',o')=(k-1,n)$ would require $\alpha = 0$, a contradiction. If $s=k$, then $(s',o')=(k-1,n)$ would require $\alpha = k$ and $o=\alpha=k$, thus $\p=(k,k)$. We also know that in this case the maximum value $p_n$ is odd, so there are at most $k$ maxima. If we use the non-reduction move of playing on all the odd stacks (and hence all maxima), then $\p'=(0,0)\in \P\setminus Q$. Thus, the second condition is fulfilled.

\item[(3)] As for condition \ref{it:GameExtI3}, from the unique position $\p = (n-2,n) \in Q$, the non-reduction move on all $n$ stacks  results in  $\p'=(2n-4,0) \in \P\setminus Q$. By Lemma~\ref{lem:redcondgenA}, play on all $n$ stacks only may require reduction when  $\sp=(n-1)p_n$, so $s\in\{0,k\}$, which is not the case here.
\end{enumerate}
Thus, by Proposition~\ref{prop:GameExtended}, the $\P$-positions of the game \SN{n}{\{n-1,n\}} are those stated in Theorem~\ref{thm:P-pos_ex_AL}.
\end{proof}

\section{Conclusion and Future Work}\label{sec:open}

We introduced a general family of games, \ruleset{Slow$A$-Nim}  $\SN{n}{A}$, where $n$ is the number of stacks in the game and $A \subset \{1,2, \ldots,n\}$ is a set that indicates the number of stacks that can be selected in each move. This family includes games that were previously studied, such as 
\ruleset{Slow Exact $k$-Nim} with $k=n-1$. We derived the \P-positions for this non-trivial family of games, complementing prior results on an optimal strategy for the family given in~\cite{GMMV2023}. We also determined the \P-positions of a new family of games, namely \SN{n}{\{n-1,n\}}, whose \P-positions closely resemble those of \SN{n}{n-1}. In the proofs of the \P-positions of both families, we used an equivalent game played on reduced positions, with moves adjusted to result in an option that is also reduced. We also developed tools that apply for general sets $A$, namely the NIRB condition and conditions on when a reduction is needed after the move. A partial result on the general structure of the $\P$-positions highlights the role of the number of odd stacks in the description of the outcome classes of \ruleset{Slow$A$-Nim} games. These general tools will be useful in the further study of these games. Among the \ruleset{Slow$A$-Nim} games that do not require reduction because no unplayable tokens exist ($1 \in A)$, we solved the game $A=\{1,k\}$ for odd values of $k$ and the case when $k=n$ even, and gave results for the game $A=\{1,2\}$ for $n=3,4,5$.

We now turn to future work. Proposition~\ref{prop:Game_n_A_Strong} and computational evidence lead to the conjecture that the structure of Theorem~\ref{thm:P-pos_ex_AL} generalizes for play on at least $k$ stacks.

\begin{conjecture} \label{con:atleastk}
For the game $\SNr{n}{A}$ with  $A=\{k,k+1,\ldots,n\}$, let $s = \sp \bmod (2k)$ and  $o$ be the number of stack heights that are odd. Then for all positions in the game, $s$ and $o$  have the same parity, and the set of \P-positions is given by  
$$\begin{aligned}
  \P_{n,A} = & \{(s,o) \mid s,o \text{ same parity}, \, 0 \leq s < k-1, \, o\leq s\}\cup \{(k-1,o) \mid  o,n \text{ same parity}, \, 1 \leq o \leq k-1\}    \\
  & \cup\{(s,o)\mid s,o \text{ same parity},\,  k-1 < s < 2k-1, \, o \le 2(k-1)-s\}=:S_1\cup S_2 \cup S_3.
\end{aligned}$$
\end{conjecture}

Figure~\ref{fig:atleastk} shows the position grids for $A=\{6,7,8,9\}$ (conjectured) and for $A=\{8,9\}$ (proven). Note that the columns for which $o \in A$ are $\N$-positions, and that the modulus of $2k$, where $k=\min(A)$, shrinks the position grid in a way that we still have the same triangular structure, just on a smaller scale. 

\begin{figure}[thb]
     \centering
     \begin{subfigure}[b]{0.35\textwidth}
         \centering
         \includegraphics[width=.75\textwidth]{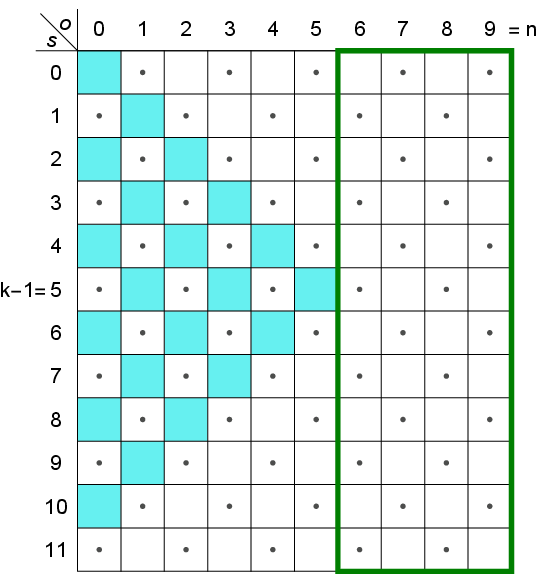}
         \caption{$\P$-positions of $\SN{9}{\{6,7,8,9\}}$.}
         \label{fig:A6-9}
     \end{subfigure}
     \hspace{0.1\textwidth}
     \begin{subfigure}[b]{0.35\textwidth}
         \centering
         \includegraphics[width=.75\textwidth]{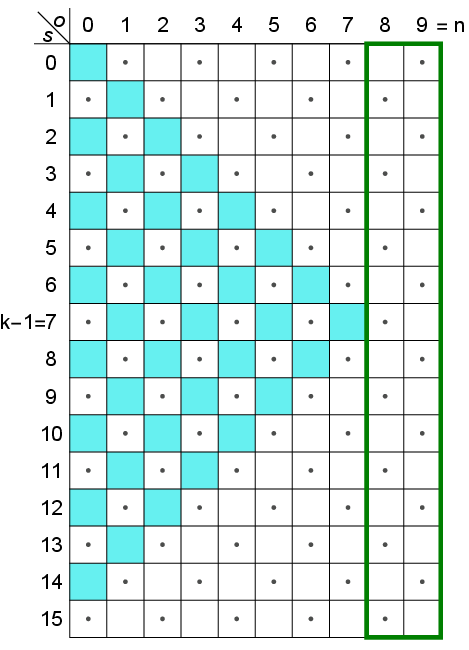}
         \caption{$\P$-positions of $\SN{9}{\{8,9\}}$.}
         \label{fig:A8-9}
     \end{subfigure}
        \caption{Conjectured structure of the \P-positions when play is on at least $k$ stacks, illustrated for $n=9$ and $k=6$ and $k=8$ (proven), respectively.}
        \label{fig:atleastk}
\end{figure}

We also propose a number of open questions. For \ruleset{Slow Exact $k$-Nim}, the next obvious choice for determining \P-positions are the families for $k=2$ and $k=n-2$. Partial results for the case $k=2$ have been given by \cite{GHHC20} for $n=3,4$ and by~\cite{ChiGurKno2021} for $n=5,6$. No results have been shown for the case $k=n-2$.

\begin{open} 
What are the \P-positions of the families \SN{n}{2}
 and \SN{n}{n-2}, and more generally, for \SN{n}{k} with $2 \leq k \leq n-2$?    
\end{open}

With regard to more general (non-singleton) sets $A$,  Theorem~\ref{thm:1k} with $k=n$ shows that the game \SN{n}{\{1,n\}} is solved. What about the game \SN{n}{\{k,n\}} for other values of $k$?

\begin{open}
    What are the \P-positions for the set $A=\{k,n\}$ for $2 \leq k\leq n-2$?
\end{open} 

Theorems~\ref{thm:1k} and~\ref{thm:12_some_} give rise to open problems regarding even values of $k$ for the set $A=\{1,k\}$. Even the case $k=2$ proves to be difficult; we have only been able to find patterns for the \P-positions when $n \leq 5$. By Proposition~\ref{prop:Game_n_A_Strong}, we know that positions where all stacks are even belong to \P~for any $n$, but not much else can be said in general.

\begin{open}
    What are the \P-positions of \SN{n}{A} for the set $A=\{1,2\}$ for $n \geq 6$?
\end{open}

\begin{open}
    What are the \P-positions of \SN{n}{A} for the set $A=\{1,k\}$ for even $k$ with $2 < k < n $?
\end{open}

 From Theorems~\ref{thm:P-pos_ex_k-1} and~\ref{thm:P-pos_ex_AL}, we know that the \P-positions of \SN{n}{\{n-1,n\}} are a subset of the \P-positions of \SN{n}{n-1}. Also, since the \P-positions of \SN{n}{1} consist of the positions for which $\sp$ is even (since exactly one token is removed in every move), we have by Theorem~\ref{thm:1k} that the \P-positions of \SN{n}{\{1,n\}} are a subset of the \P-positions of \SN{n}{1}. This leads to the following question. 

\begin{open}
    For which values of $1<k<n$ is the  set of \P-positions of \SN{n}{A} with $A=\{k,n\}$ a subset of the \P-positions of the game  with $A=\{k\}$? 
\end{open}

We look forward to progress on any of these research questions. We also thank the anonymous referee for very careful reading and detailed feedback that improved the structure and clarity of the paper.

\end{document}